\documentclass[11pt,a4paper]{article}
\usepackage{amsfonts,amsgen,amsmath,amstext,amsbsy,amsopn,amsthm,amsfonts,amssymb,amscd}
\usepackage{epsf,epsfig}
\usepackage{float}
\usepackage{ebezier,eepic}
\usepackage{color}
\usepackage{tikz}
\usepackage{multirow}
\usepackage{mathrsfs}
\usepackage{graphicx}
\usepackage{subfigure}
\newtheorem{thm}{Theorem}[section]

\newtheorem{false statement}{False statement}
\newtheorem{cor}[thm]{Corollary}

\theoremstyle{definition}

\makeatletter \@addtoreset{equation}{section}

\newcommand{\ex}{{\rm ex}}
\def\hh{\mathcal{H}}

\def\hg{\mathcal{G}}

\def\hm{\mathcal{M}}
\def\hn{\mathbb{N}}

\def\ex{\mathbb{E}}
\begin{document}
\title{\bf\Large A Note on Minimum Degree Condition for Hamilton $(a,b)$-Cycles in Hypergraphs}
\date{}
\author{Jian Wang\\[10pt]
Department of Mathematics\\
Taiyuan University of Technology\\
Taiyuan 030024, P. R. China\\[6pt]
E-mail:  wangjian01@tyut.edu.cn
}

\maketitle

\begin{abstract}
Let $k,a,b$ be positive integers with $a+b=k$. A $k$-uniform hypergraph is called an $(a,b)$-cycle if there is a partition $(A_0,B_0,A_1,B_1,\ldots,A_{t-1},B_{t-1})$ of the vertex set with $|A_i|=a$, $|B_i|=b$ such that $A_i\cup B_i$ and $B_i\cup A_{i+1}$ (subscripts module $t$) are edges for all $i=0,1,\ldots,t-1$. Let $\hh$ be a $k$-uniform $n$-vertex hypergraph with $n\geq 5k$ and $n$ divisible by $k$. By applying the concentration inequality for intersections of a uniform hypergraph with a random matching developed by Frankl and Kupavskii, we show that if  there exists $\alpha\in (0,1)$ such that $\delta_a(\hh)\geq (\alpha +o(1))\binom{n-a}{b}$ and $\delta_b(\hh)\geq (1-\alpha +o(1))\binom{n-b}{a}$, then $\hh$ contains a Hamilton $(a,b)$-cycle. As a corollary, we prove that if $\delta_{\ell}(\hh)\geq (1/2 +o(1))\binom{n-\ell}{k-\ell}$ for some $\ell \geq k/2$, then $\hh$ contains a Hamilton $(k-\ell,\ell)$-cycle and this is asymptotically best possible.
\end{abstract}

\section{Introduction}

Let $[n]=\{1,2,\ldots,n\}$ and $\binom{[n]}{k}$ be the collection of all $k$-subsets of $[n]$. A subfamily $\hh$ of $\binom{[n]}{k}$ is called a {\it $k$-uniform hypergraph} with the vertex set $[n]$ and the edge set $\hh$. Given  $\hh\subseteq \binom{[n]}{k}$ and a $d$-subset $S$ of $V$, let $\deg_{\hh}(S)$ denote  the number of edges of $\hh$ containing $S$. The {\it minimum $d$-degree} $\delta_d(\hh)$ of $\hh$ is the minimum of $\deg_{\hh}(S)$ over all $d$-subsets $S$ of $[n]$.

Let $G$ be a simple graph on $n$ vertices with $n\geq 3$. A {\it Hamilton cycle} of $G$ is a cyclic ordering of all its vertices so that any two consecutive vertices are connected by an edge. In 1952, Dirac \cite{dirac} proved that if $\delta(G)\geq n/2$ then $G$ contains a Hamilton cycle, which is one of the most classical results in graph theory. In recent years, hypergraph generalisations of Dirac's theorem are well studied and we refer to \cite{kuhn-osthus, rodl, zhao} for surveys on this topic.

Let $n,k,\ell$ be positive integers with $\ell<k$ and $(k-\ell)|n$. A $k$-uniform hypergraph is called an {\it $\ell$-cycle} if there is a cyclic ordering of the vertices such that every edge consists of $k$ consecutive vertices, every vertex is contained in an edge and two consecutive edges (where the ordering of the edges is inherited from the ordering of the vertices) intersect in exactly $\ell$-vertices. We say a $k$-uniform hypergraph $\hh$ contains a Hamilton $\ell$-cycle if there is a subhypergraph of $\hh$ which forms an $\ell$-cycle, which covers all vertices of $\hh$. Confirming a conjecture of Katona and Kierstead \cite{katona}, R\"{o}dl, Ruci\'{n}ski and Szemer\'{e}di \cite{rodlrs-1,rodlrs-2} showed that every $k$-uniform $n$-vertex hypergraph $\hh$ with $\delta_{k-1}(\hh)\geq n/2+o(n)$ contains a Hamilton $(k-1)$-cycle. This is best possible up to the $o(n)$ term by a construction given by Katona and Kierstead \cite{katona}. Very recently,  H\`{a}n, Han and Zhao \cite{hanhanzhao} determined the exact minimum $d$-degree condition that guarantees the existence of a Hamilton $(k/2)$-cycle in every $k$-uniform hypergraph on $n$ vertices for every even $k\geq 6$, $k/2\leq d \leq k-1$ and sufficiently large $n\in (k/2)\hn$.

In this note, we introduce a  different notion of hypergraph cycles.  To the best of our knowledge, it is  first considered.   Let $k,a,b$ be positive integers with $k=a+b$. A $k$-uniform hypergraph is called an $(a,b)$-cycle if there is a partition $(A_0,B_0,A_1,B_1,\ldots,A_{t-1},B_{t-1})$ of the vertex set with $|A_i|=a$, $|B_i|=b$ such that $A_i\cup B_i$, $B_i\cup A_{i+1}$ (subscripts module $t$) are edges for all $i=0,1,\ldots,t-1$.  We say a $k$-uniform, $n$-vertex hypergraph $\hh$ contains a Hamilton $(a,b)$-cycle if there is a subhypergraph of $\hh$ which forms an $(a,b)$-cycle and which covers all vertices of $\hh$. Note that a Hamilton $(k-1)$-cycle always contains a Hamilton $(a,b)$-cycle. Thus,  R\"{o}dl, Ruci\'{n}ski and Szemer\'{e}di's result \cite{rodlrs-2} implies that every $k$-uniform $n$-vertex hypergraph $\hh$ with $\delta_{k-1}(\hh)\geq n/2+o(n)$ contains a Hamilton $(a,b)$-cycle for every $a,b\geq 1$ with $a+b=k$. For even $k$, a Hamilton $(k/2,k/2)$-cycle is exactly a Hamilton $(k/2)$-cycle. Thus, H\`{a}n, Han and Zhao's result \cite{hanhanzhao} implies the exact  minimum $d$-degree condition that guarantees the existence of a Hamilton $(a,b)$-cycle with $a=b=k/2$ for $k/2\leq d\leq k$ and even $k$.

By applying the concentration inequality for intersections of a uniform hypergraph with a random matching in \cite{frankl-kupavskii}, we prove the following result.

\begin{thm}\label{main-1}
 Let $n, k, a, b$ be positive integers with $n\in k\hn$, $n\geq 5k$ and $k=a+b$. Let $\hh\subset \binom{[n]}{k}$. If there exists $\alpha\in (0,1)$ such that  $\delta_a(\hh)\geq (\alpha +4\sqrt{k\log n/ n})\binom{n-a}{b}$ and $\delta_b(\hh)\geq (1-\alpha +4\sqrt{k\log n/ n})\binom{n-b}{a}$, then $\hh$ contains a Hamilton $(a,b)$-cycle.
\end{thm}

If $a\leq b$ and $\delta_b(\hh)\geq \beta\binom{n-b}{a}$, then we also have
\[
\delta_a(\hh) \geq \frac{1}{\binom{b}{b-a}}\binom{n-a}{b-a} \delta_b(\hh)\geq \beta\frac{\binom{n-a}{b-a}\binom{n-b}{a}}{\binom{b}{b-a}} = \beta \binom{n-a}{b}.
\]
Thus Theorem \ref{main-1} implies the following corollary.

\begin{cor}\label{cor-1}
Let $n, k,\ell$ be integers with $n\in k\hn$, $n\geq 5k$ and $k/2\leq \ell\leq k-1$. Let $\hh\subset \binom{[n]}{k}$. If $\delta_{\ell}(\hh)\geq (1/2 +4\sqrt{k\log n/ n})\binom{n-\ell}{k-\ell}$, then $\hh$ contains a Hamilton $(k-\ell,\ell)$-cycle.
\end{cor}

It is easy to see that a Hamilton $(k-\ell,\ell)$-cycle of $\hh$ can be decomposed into two perfect matchings. Hence the same minimum $\ell$-degree condition in Corollary \ref{cor-1} also guarantees the existence of a perfect matching. Pikhurko \cite{pikhurko} showed that the major term of the asymptotically tight minimum $\ell$-degree to guarantee a perfect matching in an $n$-vertex $k$-uniform hypergraph is $\frac{1}{2}\binom{n-\ell}{k-\ell}$ for $k/2\leq \ell \leq k-1$ and $n\in k\hn$. It follows that the minimum $\ell$-degree condition in Corollary \ref{cor-1} is asymptotically best possible. We should also mention that Treglown and Zhao \cite{treglown-zhao1, treglown-zhao2} determined the tight minimum $\ell$-degree to guarantee a perfect matching in a $k$-uniform hypergraph for $k/2\leq \ell \leq k-1$.

Let $n, a, b\geq 1$ be integers and let $V_1,V_2$ be two disjoint sets with $|V_1|=an, |V_2|=bn$. Define the direct product $\binom{V_1}{a}\sqcup\binom{V_2}{b}$ as the collection of all subsets $F$ of $V_1\cup V_2$ with $|F\cap V_1|= a$ and $|F\cap V_2|= b$. By the same argument, we also obtain the following result.

\begin{thm}\label{main-2}
Let $\hh\subset \binom{V_1}{a}\sqcup \binom{V_2}{b}$ with $|V_1|=an$, $|V_2|=bn$ and $n\geq 5$. If there exists $\alpha\in (0,1)$ such that $\delta_a(\hh)\geq (\alpha +4\sqrt{\log n/ n})\binom{bn}{b}$ and $\delta_b(\hh)\geq (1-\alpha +4\sqrt{\log n/ n})\binom{an}{a}$, then $\hh$ contains a Hamilton $(a,b)$-cycle.
\end{thm}

\section{Proof of Theorems \ref{main-1} and \ref{main-2}}

In this section, by applying the concentration inequality for intersections of a uniform hypergraph with a random matching \cite{frankl-kupavskii} and an Ore-type theorem on bipartite graph due to Moon and Moser \cite{moon-moser}, we obtain a minimum degree condition for the existence of Hamilton $(a,b)$-cycles in $k$-uniform hypergraphs.

\begin{thm}[Frankl and Kupavskii, \cite{frankl-kupavskii}]\label{thm-fk}
Suppose that $m, \ell, t$ are integers and $m \geq  t\ell$. Let $\hg \subset\binom{[m]}{\ell}$ and let $\hm$ be a matching of size $t$ chosen from $\binom{[m]}{\ell}$ uniformly at random.  Let $\eta=|\hg\cap \hm|$ and let $\theta=|\hg|/\binom{[m]}{\ell}$.  Then $\ex [\eta] = \theta t$ and, for any positive $\gamma$ we have
\[
Pr[|\eta -\theta t|\geq 2\gamma \sqrt{t}]\leq 2e^{-\gamma^2/2}.
\]
\end{thm}

\begin{thm}[Moon and Moser, \cite{moon-moser}]\label{thm-mm}
Let $G[X,Y]$ be a bipartite graph with partite sets $X,Y$ and $|X|=|Y|=n$. If $\deg(x)+\deg(y)> n$ for any non-edge $xy$ of $G$ with $x\in X$ and $y\in Y$, then $G$ contains a Hamilton cycle.
\end{thm}

\begin{proof}[Proof of Theorem \ref{main-1}.]
Let $t=n/k\geq 5$. A tuple $(A_1, \ldots,  A_t, B_1, \ldots, B_t)$ is called an {\it $(a,b)$-partition} of $[n]$ if  $A_1, \ldots,  A_t, B_1, \ldots, B_t$ are pairwise disjoint, their union is $[n]$, and $|A_i|=a$, $|B_j|=b$ for every $i=1,2,\ldots,t$ and $j=1,2,\ldots,t$. Let $\Omega$ be the set of  $(a,b)$-partitions of $[n]$ and let $(A_1, \ldots,  A_t, B_1, \ldots, B_t)$ be an $(a,b)$-partition chosen from $\Omega$ uniformly at random. Let $X=\{A_1,A_2,\ldots,A_t\}$ and  $Y=\{B_1,B_2,\ldots,B_t\}$. Consider the  bipartite graph $G[X,Y]$ where we have an edge $(A_i,B_j)$ iff $A_i\cup B_j\in \hh$. If there exists an  $(a,b)$-partition of $[n]$  so that $\deg_G(A_i)+\deg_G(B_j)> t$ holds for all pairs $(i,j)$ with $i,j\in [t]$, then by Theorem \ref{thm-mm} there is a Hamilton cycle in $G[X,Y]$, which is also a Hamilton $(a,b)$-cycle of $\hh$. Thus we are left to show that the probability that $\deg(A_i)+\deg(B_j)> t$ for all pairs $(i,j)$ is positive.

Let $A$ be a given $a$-set in $[n]$ and $B$ be a given $b$-set in $[n]$. Set $\hh[A]=\{S\colon S\cup A\in \hh\}$, $\hh[B]=\{S\colon S\cup B\in \hh\}$, $\eta_A=|\hh[A]\cap Y|$, $\eta_B=|\hh[B]\cap X|$, $\gamma = 2\sqrt{\log t}$,
$\alpha_A = |\hh[A]|/\binom{n-a}{b}$ and $\beta_B = |\hh[B]|/\binom{n-b}{a}$. Note that $\alpha_A \geq \alpha+4\sqrt{\log t/t}$. It follows that
\begin{align}\label{ineq-1}
 Pr\left[\deg_G(A_i)\leq\alpha t\right] = &\sum_{A\in \binom{[n]}{a}} Pr\left[\deg_G(A_i)\leq\alpha t \mid A_i=A\right]\cdot Pr[A_i=A]\nonumber\\[4pt]
= &\binom{n}{a}^{-1} \sum_{A\in \binom{[n]}{a}} Pr\left[\eta_A\leq\alpha t  \mid A_i=A\right]\nonumber\\[4pt]
\leq &\binom{n}{a}^{-1} \sum_{A\in \binom{[n]}{a}} Pr\left[\eta_A\leq\alpha_A t-4\sqrt{t\log t}  \mid A_i=A\right].
\end{align}
Since $Y$ can be viewed as a random $t$-matching chosen uniformly from $\binom{[n]\setminus A}{b}$ under the condition $A_i=A$, by Theorem \ref{thm-fk} we have
\begin{align}\label{ineq-2}
  Pr\left[\left|\eta_A -\alpha_A t\right|\geq 4\sqrt{t\log t}\mid A_i=A \right]\leq \frac{2}{t^2}.
\end{align}
Combining \eqref{ineq-1} and \eqref{ineq-2}, we obtain that
\[
Pr\left[\deg(A_i)\leq\alpha t\right] \leq \frac{2}{t^2}.
\]
Similarly, we can show that
\[
Pr\left[\deg(B_j)\leq(1-\alpha) t\right] \leq \frac{2}{t^2}.
\]
Then by the union bound, with probability  at most $2t\cdot\frac{2}{t^2}<1$  one of the events $\deg_G(A_i)\leq\alpha t$ and $\deg(B_j)\leq(1-\alpha)t$ occurs. Thus with positive probability we have
\[
\deg_G(A_i)>  \alpha t \mbox{ and }\deg_G(B_j)> (1-\alpha)t
\]
for all $i=1,2,\ldots,t$ and $j=1,2,\ldots,t$. It follows that with positive probability
\[
\deg_{G}(A_i)+\deg_{G}(B_j)>t
\]
for all pairs $(i,j)$ and this completes the proof.
\end{proof}

\begin{proof}[Proof of Theorem \ref{main-2}.]
Let $X=\{A_1,A_2,\ldots,A_n\}$ be a perfect matching in $\binom{V_1}{a}$ and  $Y=\{B_1,B_2,\ldots,B_n\}$ be a perfect matching in $\binom{V_2}{b}$, both of them are chosen uniformly at random. Consider the  bipartite graph $G[X,Y]$ where we have an edge $(A_i,B_j)$ iff $A_i\cup B_j\in \hh$. Let $A$ be a given $a$-set in $V_1$ and $B$ be a given $b$-set in $V_2$. Set $\hh[A]=\{S\subset V_2\colon S\cup A\in \hh\}$, $\hh[B]=\{S\subset V_1\colon S\cup B\in \hh\}$, $\eta_A=|\hh[A]\cap Y|$, $\eta_B=|\hh[B]\cap X|$, $\gamma = 2\sqrt{\log n}$,
$\alpha_A = |\hh[A]|/\binom{bn}{b}$ and $\beta_B = |\hh[B]|/\binom{an}{a}$. Then by Theorem \ref{thm-fk}, we have
\[
Pr\left[\left|\eta_A -\alpha_A n\right|\geq 4\sqrt{n\log n}\right]\leq \frac{2}{n^2}\mbox{ and } Pr\left[\left|\eta_B -\beta_B n\right|\geq 4\sqrt{n\log n}\right]\leq \frac{2}{n^2}.
\]
Note that $\alpha_A \geq \alpha+4\sqrt{\log n/n}$. It follows that
\begin{align*}
 Pr\left[\deg_G(A_i)\leq\alpha n \right] = &\sum_{A\in \binom{V_1}{a}} Pr\left[\deg_G(A_i)\leq\alpha n \mid A_i=A\right]\cdot Pr[A_i=A]\\[4pt]
= &\binom{an}{a}^{-1} \sum_{A\in \binom{V_1}{a}} Pr\left[\eta_A\leq\alpha_A n-4\sqrt{n\log n} \right]\\[4pt]
\leq &\binom{an}{a}^{-1} \sum_{A\in \binom{V_1}{a}} \frac{2}{n^2} = \frac{2}{n^2}.
\end{align*}
By the same argument, we have
\[
Pr\left[\deg_G(B_j)\leq(1-\alpha) n \right] \leq \frac{2}{n^2}.
\]
Then by the union bound, with positive probability $\deg_G(A_i)+\deg_G(B_j)>  n$ for all $i=1,2,\ldots,n$ and $j=1,2,\ldots,n$. Thus, there exist perfect matchings $X$ and $Y$ so that $\deg_G(A_i)+\deg_G(B_j)> n$ holds for all pairs $(i,j)$ with $i,j\in [n]$. By Theorem \ref{thm-mm} the theorem follows.
\end{proof}

\section{Concluding Remarks}

Instead of the usual absorption approach, we use a concentration inequality for the size of intersections of a uniform hypergraph and a random matching developed in \cite{frankl-kupavskii}. By a probabilistic argument, a Hamilton $(a,b)$-cycle is transferred to a Hamilton cycle in ordinary bipartite graphs. By an Ore-type condition on bipartite graphs due to Moon and Moser \cite{moon-moser}, we give a minimum degree condition that guarantees a Hamilton $(a,b)$-cycle in $k$-uniform hypergraphs.  Specifically, for $\ell\geq k/2$ the asymptotically tight minimum $\ell$-degree that guarantees a Hamilton $(k-\ell,\ell)$-cycle is given in Corollary \ref{cor-1}. Note that H\`{a}n, Han and Zhao \cite{hanhanzhao} determined the exact minimum $(k/2)$-degree condition that guarantees the existence of a Hamiltion $(k/2,k/2)$-cycle. It seems an interesting problem to determine the exact minimum $\ell$-degree condition that guarantees a Hamilton $(k-\ell,\ell)$-cycle for $\ell> k/2$.

\vspace{6pt}\noindent
{\bf Acknowledgement:} We thank the referees for their helpful comments. The work was
supported by the National Natural Science Foundation of China (No. 11701407).

\end{document}